\documentclass[12pt]{article}

\usepackage{amsfonts}
\usepackage{amssymb}
\usepackage{amsthm}
\usepackage{amsmath}
\usepackage{graphicx}
\usepackage[lofdepth,lotdepth]{subfig}
\usepackage{times}
\usepackage{rotating}

\newtheorem{thm}{Theorem}[section]
\newtheorem{prop}[thm]{Proposition}
\newtheorem{corollary}[thm]{Corollary}
\newtheorem{lem}[thm]{Lemma}

\theoremstyle{definition}
\newtheorem{example}[thm]{Example}
\newtheorem{definition}[thm]{Definition}
\newtheorem{remark}[thm]{Remark}

\newtheorem*{ack}{Acknowledgements}

\DeclareMathOperator{\im}{im}
\DeclareMathOperator{\facets}{Facets}
\DeclareMathOperator{\Z}{\mathbb Z}
\DeclareMathOperator{\N}{\mathcal{N}}
\DeclareMathOperator{\F}{\mathcal{F}}

\begin{document}

\title{Chorded complexes and a necessary condition for a monomial ideal to have a linear resolution}

\author{ E. Connon\thanks{Research supported by NSERC and Killam scholarships.} \qquad \qquad S. Faridi\thanks{Research supported by NSERC.} }

\maketitle

\begin{center} \it \small
Department of Mathematics and Statistics, Dalhousie University, Canada
\end{center}

\begin{abstract}
In this paper we extend one direction of Fr\"oberg's theorem on a combinatorial classification of quadratic monomial ideals with linear resolutions.  We do this by generalizing the notion of a chordal graph to higher dimensions with the introduction of $d$-chorded and orientably-$d$-cycle-complete simplicial complexes.  We show that a certain class of simplicial complexes, the $d$-dimensional trees, correspond to ideals having linear resolutions over fields of characteristic $2$ and we also give a necessary combinatorial condition for a monomial ideal to be componentwise linear over all fields.
\end{abstract}



\section{Introduction} \label{sec:intro}

One approach to studying algebraic properties of square-free monomial ideals is to examine the combinatorics of their associated simplicial complexes or hypergraphs.  To any square-free monomial ideal $I$ one can associate both a {\bf Stanley-Reisner complex} whose faces are the monomials not in $I$ and a {\bf facet complex} whose facets are the minimal generators of $I$.  One can also consider the facet complex as a hypergraph whose {\bf edge ideal} is $I$.  It turns out that specific algebraic characteristics often correspond to complexes or hypergraphs with a well-defined combinatorial structure. See, for example, \cite{Em10}, \cite{Far04}, \cite{FrVT07}, \cite{HaVT08}, \cite{HHZ06} and \cite{Wood11}.

In 1990, Fr\"oberg gave a complete combinatorial classification of the square-free monomial ideals having $2$-linear resolutions \cite{Fr90}.  He showed that the edge ideal of a graph has a linear resolution if and only if the complement of the graph is chordal.  Since then, finding a generalization of this theorem to square-free monomial ideals of higher dimensions has been an active area of research.  See, for example, \cite{Em10}, \cite{MNYZ12}, \cite{MYZ12}, and \cite{Wood11}.  In \cite{EagRein98}, Eagon and Reiner showed that the property of having a linear resolution is dual to the Cohen-Macaulay property.  Therefore finding a generalization of Fr\"oberg's theorem would provide a combinatorial classification, via the Alexander dual, of those square-free monomial ideals which are Cohen-Macaulay.  In fact, using the technique of polarization one would obtain a classification of all Cohen-Macaulay monomial ideals.  See \cite{Far07} for a description of this technique.

Unfortunately, it is too much to expect that the property of having a linear resolution can be described purely through a combinatorial property of an associated combinatorial structure as the existence of such a resolution does depend on the field over which the polynomial ring is defined.  The triangulation of the real projective plane is a typical example.  In this instance, the facet ideal of the complement of this complex has a linear resolution only when the characteristic of the field in question is not equal to $2$.  However, it is not unreasonable to expect that such a complete combinatorial classification may exist for ideals having linear resolutions over all fields.

The general approach to finding a generalization of Fr\"oberg's theorem has been to extend the definition of a chordal graph to higher dimensions.  In \cite{Em10}, \cite{HaVT08}, and \cite{Wood11} different notions of a ``chordal'' hypergraph are presented.  In \cite{Em10} and \cite{Wood11} Emtander and Woodroofe respectively prove that the edge ideal of the complement of a chordal hypergraph (under their respective definitions) has a linear resolution over all fields.  However, as we will see in Section \ref{sec:necessary_condition}, the converses of their theorems do not hold.

In our approach to extending Fr\"oberg's theorem we introduce the class of {\bf $d$-chorded} simplicial complexes and the class of {\bf orientably-$d$-cycle-complete} complexes.  In Section \ref{sec:necessary_condition} we prove the following theorem, which includes a generalization of one direction of Fr\"oberg's theorem and a necessary condition for an ideal to have a linear resolution.

\begin{thm} \label{thm:converse_intro}
Let $\Gamma$ be a simplicial complex.  If the facet ideal of $\Gamma$ has a $(d+1)$-linear resolution over $k$ then
\begin{enumerate}
\item the $d$-complement of $\Gamma$ is orientably-$d$-cycle-complete;
\item the $d$-complement of $\Gamma$ is $d$-chorded if $k$ has characteristic $2$.
\end{enumerate}
\end{thm}

Theorem \ref{thm:converse_intro} can also be stated in terms of Stanley-Reisner ideals.  It appears this way in Section \ref{sec:necessary_condition} as Theorem \ref{thm:one_dir}.

Our definitions of $d$-chorded complexes and orientably-$d$-cycle-complete complexes use the notion of a {\bf $d$-dimensional cycle} in a simplicial complex.  This concept was introduced by the first author in \cite{Con13} as a higher-dimensional notion of a graph cycle.

The idea of a $d$-dimensional cycle leads naturally to the concept of a {\bf $d$-dimensional tree}, in which $d$-dimensional cycles are forbidden.  These simplicial complexes correspond to ideals with linear resolutions over fields of characteristic $2$.

\begin{thm} \label{thm:trees_intro}
The facet ideal of the $d$-complement of a $d$-dimensional tree has a $(d+1)$-linear resolution over any field of characteristic $2$.
\end{thm}

Theorem \ref{thm:trees_intro} also implies that the facet ideal of the $d$-complement of a pure $d$-dimensional {\bf simplicial tree} (\cite{Far02}) has a linear resolution over fields of characteristic $2$.

For monomial ideals whose generators are not all of the same degree, the analogous notion to having a linear resolution is the property of being componentwise linear.  The notion of a {\bf chorded} simplicial complex extends the idea of a $d$-chorded complex to the non-pure case and we are able to show the following theorem.

\begin{thm} \label{thm:cwl_intro}
If the Stanley-Reisner ideal of the simplicial complex $\Gamma$ is componentwise linear over every field $k$ then $\Gamma$ is chorded.
\end{thm}

Theorem \ref{thm:converse_intro} is a special case of Theorem \ref{thm:cwl_intro}.

This paper is organized as follows.  In the next section we review some basic definitions from simplicial homology, Stanley-Reisner theory, and commutative algebra as well as discussing Fr\"oberg's original theorem on square-free monomial ideals with $2$-linear resolutions.  In Section \ref{sec:d_dim_cycles} we discuss $d$-dimensional cycles and their structure.  In Section \ref{sec:dchorded_complexes} we introduce $d$-chorded complexes, $d$-cycle-complete complexes and $d$-dimensional trees.  In Section \ref{sec:homology_ddim_cycles} we discuss the simplicial homology of $d$-dimensional cycles and of related structures.  In Section \ref{sec:necessary_condition} we prove Theorem \ref{thm:converse_intro} and in Section \ref{sec:does_chorded_have_LR} we discuss its converse and prove Theorem \ref{thm:trees_intro}.  Finally, in Section \ref{sec:chorded} we introduce the notion of a chorded simplicial complex and prove Theorem \ref{thm:cwl_intro}.

\begin{ack}
The authors would like to thank MSRI for their hospitality during the preparation of this paper.
\end{ack}


\section{Preliminaries}

\subsection{Simplicial complexes and simplicial homology} \label{sec:simp_cmpx_hom}

An (abstract) {\bf simplicial complex} $\Gamma$ on the finite {\bf vertex set} $V$ is a set of subsets of $V$ such that $\{v\} \in \Gamma$ for all $v \in V$ and for any $F \in \Gamma$ if $G \subseteq F$ then $G \in \Gamma$.  The elements of $V$ are {\bf vertices} of $\Gamma$ and the elements of $\Gamma$ are called {\bf faces} or {\bf simplices} of $\Gamma$.  Faces of $\Gamma$ that are maximal with respect to inclusion are called {\bf facets} of $\Gamma$ and we use the notation $\facets(\Gamma)$ for this set of faces.  We denote the vertex set of $\Gamma$ by $V(\Gamma)$.  If $\facets(\Gamma)=\{F_1,\ldots,F_k\}$ then we write
\[
    \Gamma = \langle F_1,\ldots, F_k \rangle.
\]
If $F$ is a face of $\Gamma$ then the {\bf dimension} of $F$ denoted by $\dim F$, is $|F|-1$ while the dimension of $\Gamma$ itself is
\[
    \dim \Gamma = \max\{\dim F: F \in \Gamma\}.
\]
A face of dimension $n$ is called an {\bf $n$-face} or an {\bf $n$-simplex}.  If all facets of $\Gamma$ have the same dimension then $\Gamma$ is said to be {\bf pure}.  The {\bf $d$-complement} of a pure $d$-dimensional simplicial complex $\Gamma$ is the complex on $V(\Gamma)$ whose facets are the $(d+1)$-subsets of $V(\Gamma)$ that are not $d$-faces of $\Gamma$.  A simplicial complex $\Gamma$ is said to be {\bf $d$-complete} if all possible subsets of $V(\Gamma)$ of size $d+1$ are faces of $\Gamma$.  The {\bf complete graph} on $n$ vertices $K_n$ is a $1$-complete simplicial complex.

The {\bf pure $d$-skeleton} of a simplicial complex $\Gamma$, denoted $\Gamma^{[d]}$, is the simplicial complex on the same vertex set as $\Gamma$ whose facets are the $d$-faces of $\Gamma$.  A {\bf subcomplex} of $\Gamma$ is any simplicial complex whose set of facets is a subset of the faces of $\Gamma$.  Given any $W \subseteq V(\Gamma)$ the {\bf induced subcomplex of $\Gamma$ on $W$} is the complex
\[
    \Gamma_W=\{F \in \Gamma | F \subseteq W \}.
\]

To any $d$-face in a simplicial complex we can assign an {\bf orientation} by specifying an ordering of its vertices.  Two orientations are said to be {\bf equivalent} if one is an even permutation of the other.  Thus there are only two equivalence classes of orientations for each face.  By an {\bf oriented $d$-face} we mean a $d$-face with a choice of one of these orientations.  We denote the $d$-face on vertices $v_0,\ldots,v_d$ with the orientation $v_0 < \cdots < v_d$ by $[v_0,\ldots,v_d]$.  We will also need the concept of {\bf induced orientation} of a face in the simplicial complex.

\begin{definition}[{\bf induced orientation}]
Given an orientation of a $d$-face in a simplicial complex the {\bf induced orientation} of any $(d-1)$-subface is given by the following procedure, where $v_0$ is considered to be in an even position:
\begin{itemize}
\item if the vertex removed to obtain the $(d-1)$-face was in an odd position of the ordering then the orientation of the $(d-1)$-face is the same as the ordering of its vertices in the $d$-face
\item if the vertex removed to obtain the $(d-1)$-face was in an even position of the ordering then the orientation of the $(d-1)$-face is given by any odd permutation of the ordering of the vertices in the $d$-face
\end{itemize}
\end{definition}

Let $A$ be a commutative ring with unit.  Then we define $C_d(\Gamma)$ to be the free $A$-module whose basis is the oriented $d$-faces of $\Gamma$ and where $[v_0,v_1,\ldots,v_d]=-[v_1,v_0,\ldots,v_d]$.  The elements of $C_d(\Gamma)$ are called {\bf $d$-chains}.  The {\bf support complex} of a $d$-chain is the complex whose facets are the $d$-faces in the $d$-chain whose coefficients are non-zero.

There is a natural boundary map homomorphism $\partial_d$ from the space of $d$-chains to the space of $(d-1)$-chains defined by setting
\[
    \partial_d([v_0,\ldots,v_d]) = \sum_{i=0}^{d}(-1)^i [v_0,\ldots,v_{i-1},v_{i+1},\ldots,v_d]
\]
for each oriented $d$-face $[v_0,\ldots,v_d]$.  The kernel of $\partial_d$ is called the group of {\bf $d$-cycles} and the image of $\partial_d$ is called the group of {\bf $(d-1)$-boundaries}.   The {\bf $d$th simplicial homology group} of $\Gamma$ over $A$ is equal to the quotient of the group of $d$-cycles over the group of $d$-boundaries and is denoted  $H_d(\Gamma;A)$.  Roughly speaking, a non-zero element of $H_d(\Gamma;A)$ indicates the presence of a ``$d$-dimensional hole'' in the complex.  For example, when $\Gamma$ is a triangulation of a sphere then $H_2(\Gamma;\Z)$ is non-zero.

We can also define a homomorphism
$\epsilon: C_0(\Gamma) \rightarrow A$
by $\epsilon(v)=1$ for each $v \in V(\Gamma)$. The {\bf reduced homology group} of $\Gamma$ in dimension $0$, denoted $\tilde{H}_0(\Gamma;A)$, is defined by
\[
    \tilde{H}_0(\Gamma;A) = \ker \epsilon / \im \partial_1.
\]
We set $\tilde{H}_i(\Gamma;A)=H_i(\Gamma;A)$ for $i>0$ to obtain the {\bf$i$th  reduced homology group} of $\Gamma$.  For a more detailed description of simplicial homology see \cite{Munk84}.


\subsection{Stanley-Reisner ideals and facet ideals}

To any simplicial complex $\Gamma$ on the vertex set $\{x_1,\ldots,x_n\}$ we can associate, in two different ways, a square-free monomial ideal in the polynomial ring $R = k[x_1,\ldots,x_n]$ where $k$ is a field.  Given a subset $F = \{x_{i_1},x_{i_2},\ldots,x_{i_k}\}$ of $V(\Gamma)$ we define $x^F$ to be the monomial $x_{i_1}x_{i_2}\cdots x_{i_k}$ in $R$. The {\bf Stanley-Reisner ideal} of $\Gamma$ is the ideal
\[
    \N(\Gamma) = \left( \{x^F: F \notin \Gamma\} \right)
\]
and the {\bf facet ideal} of $\Gamma$ (or the {\bf edge ideal} if we think of $\Gamma$ as a hypergraph) is the ideal
\[
    \F(\Gamma) = \left( \{x^F: F \in \Gamma \} \right).
\]
The {\bf Stanley-Reisner ring} of $\Gamma$ is the ring $k[\Gamma]=R/\N(\Gamma)$.

The {\bf Stanley-Reisner complex} of the square-free monomial ideal $I$ is the complex $\N(I)$ whose faces are given by the square-free monomials not in $I$.  The {\bf facet complex} of $I$ is the complex $\F(I)$ whose facets are given by the minimal monomial generators of $I$.  See Figure \ref{fig:SR_complex_new} for examples of these relationships.

\begin{figure}[h!]
{\centering
    \includegraphics[height=1.15in]{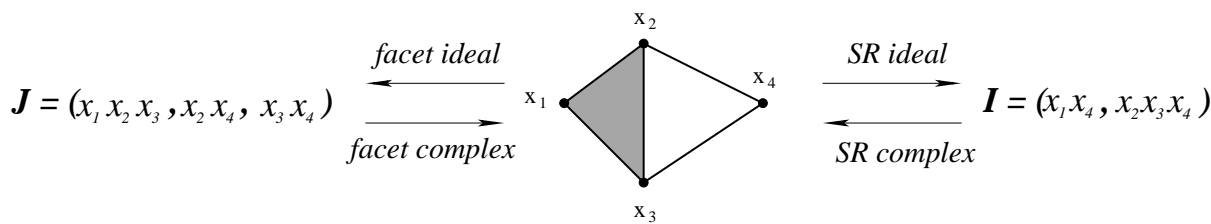}
    \caption {Relationship between simplicial complexes and ideals.} \label{fig:SR_complex_new}

}
\end{figure}


\subsection{Linear resolutions and Fr\"oberg's Theorem} \label{sec:F_orig_thm}
The monomial ideal $I$ in the polynomial ring $R=k[x_1,\ldots,x_n]$ is said to have a {\bf $d$-linear resolution over $k$}, or simply a {\bf linear resolution over $k$}, if all generators of $I$ have degree $d$ and in a minimal graded free resolution of $I$
{
\[
	0 \rightarrow \bigoplus_j R(-j)^{\beta_{k,j}(I)} \rightarrow \bigoplus_j R(-j)^{\beta_{k-1,j}(I)} \rightarrow \cdots \rightarrow \bigoplus_j R(-j)^{\beta_{0,j}(I)} \rightarrow I \rightarrow 0
\]
}
we have $\beta_{i,j}(I)=0$ for all $j \neq i+d$.

In 1990, Fr\"oberg gave a characterization of ideals with 2-linear resolutions in terms of the combinatorial structure of an associated graph.

Recall that a {\bf graph cycle} is a sequence of adjacent vertices in a graph in which the vertices are distinct and the last vertex in the sequence is adjacent to the first vertex of the sequence.  A graph $G$ is called {\bf chordal} if all cycles in $G$ of length greater than three have a {\bf chord}, where a chord of a cycle is an edge between non-adjacent vertices of the cycle.  The {\bf complement} of a graph $G$ is the graph on the same vertex set as $G$ but whose edges are exactly those $2$-sets that are not edges of $G$.

Given a graph $G$ we can obtain a simplicial complex $\Delta(G)$, called the {\bf clique complex} of $G$, by taking the sets of vertices of complete subgraphs of $G$ as the faces of $\Delta(G)$.

\begin{thm}[Fr\"oberg \cite{Fr90}] \label{thm:Frob_original}
If a graph $G$ is chordal, then $\N(\Delta(G))$ has a $2$-linear resolution over any field.  Conversely, if $\N(\Gamma)$ has a $2$-linear resolution over some field, then $\Gamma = \Delta(\Gamma^{[1]})$ and $\Gamma^{[1]}$ is chordal.
\end{thm}

This theorem is more commonly stated in the following way.

\begin{thm}[Fr\"oberg]
The edge ideal of a graph $G$ has a $2$-linear resolution over a field if and only if the complement of $G$ is chordal.
\end{thm}

It is not hard to confirm the equivalence of these two theorems by noticing that the edge ideal of the complement of a graph $G$ is the Stanley-Reisner ideal of the clique complex of $G$.

The proof of Fr\"oberg's theorem relies on the following homological characterization of the square-free monomial ideals with linear resolutions.

\begin{thm}[Fr\"oberg \cite{Fr85}] \label{thm:Frob_hom_char}
A square-free monomial ideal $I$ has a $t$-linear resolution over $k$ if and only if for every induced subcomplex $\Gamma$ of $\N(I)$ we have $\tilde{H}_i(\Gamma;k) = 0$ for $i \neq t-2$.
\end{thm}

The success of Fr\"oberg's proof of Theorem \ref{thm:Frob_original} greatly relies on the fact that a graph cycle is the correct combinatorial notion to capture the idea of non-zero 1-dimensional homology in a simplicial complex.  Our approach to extending Fr\"oberg's Theorem to higher dimensions involves identifying the combinatorial structures in simplicial complexes which lead to non-zero homology.  By creating a class of complexes in which this type of structure is restricted in certain ways we are able to make use of the homological classification given in Theorem \ref{thm:Frob_hom_char}.


\section{$d$-dimensional cycles} \label{sec:d_dim_cycles}
Before introducing the notion of a higher-dimensional cycle, we need the following definitions.

\begin{definition}[{\bf $d$-path, $d$-path-connected, $d$-path-connected components}]
A sequence $F_1,\ldots,F_k$ of $d$-faces in a simplicial complex is called a {\bf $d$-path} between $F_1$ and $F_k$ if for all $1 \leq i \leq k-1$ we have that $|F_i \cap F_{i+1}| = d$.  A pure $d$-dimensional simplicial complex is {\bf $d$-path-connected} or {\bf strongly connected} if there exists a $d$-path between each pair of its $d$-faces.  The {\bf $d$-path-connected components} of a pure $d$-dimensional simplicial complex $\Gamma$ are the maximal subcomplexes of $\Gamma$ which are $d$-path-connected.
\end{definition}

Note that two distinct $d$-path-connected components of a complex do not share any $(d-1)$-faces.  In Figure \ref{fig:dpathcomponents} we have an example of a pure $2$-dimensional simplicial complex with two $2$-path-connected components which are emphasized by different levels of shading.  The left-hand component is an example of a $2$-path between the $2$-faces $F$ and $G$.

\begin{figure}[h!]
{\centering
    \includegraphics[height=.75in]{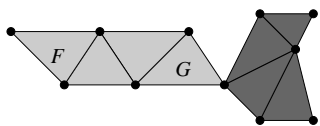}
    \caption {Example of a $2$-path and $2$-path-connected components.} \label{fig:dpathcomponents}

}
\end{figure}

\begin{definition} [Connon~\cite{Con13} {\rm{\bf $d$-dimensional cycle}}]
A pure $d$-dimensional simplicial complex $\Omega$ is called a {\bf $d$-dimensional cycle} if
\begin{enumerate}
\item $\Omega$ is $d$-path-connected, and
\item every $(d-1)$-face of $\Omega$ is contained in an even number of $d$-faces of $\Omega$.
\end{enumerate}
\end{definition}

In Figure \ref{fig:2d_cycles} we give several examples of $2$-dimensional cycles.

\begin{figure}[h]
\centering
\subfloat[A hollow tetrahedron]{\makebox[7cm]{
            \includegraphics[height=1in]{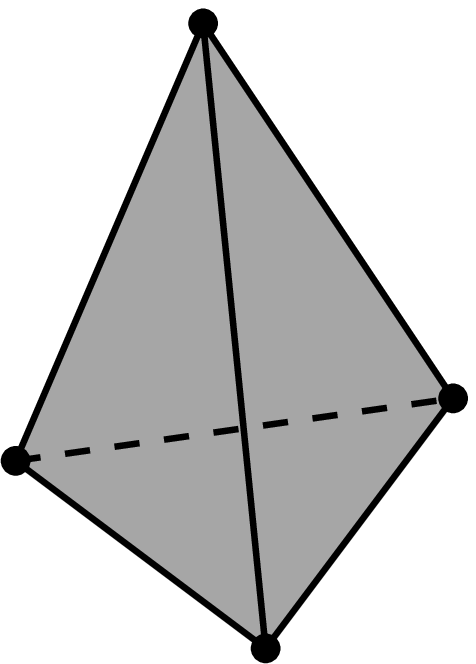}}
            \label{fig:hollow_tetra}} \qquad
\subfloat[A triangulation of the real projective plane]{\makebox[7cm]{
	\includegraphics[height=.9in]{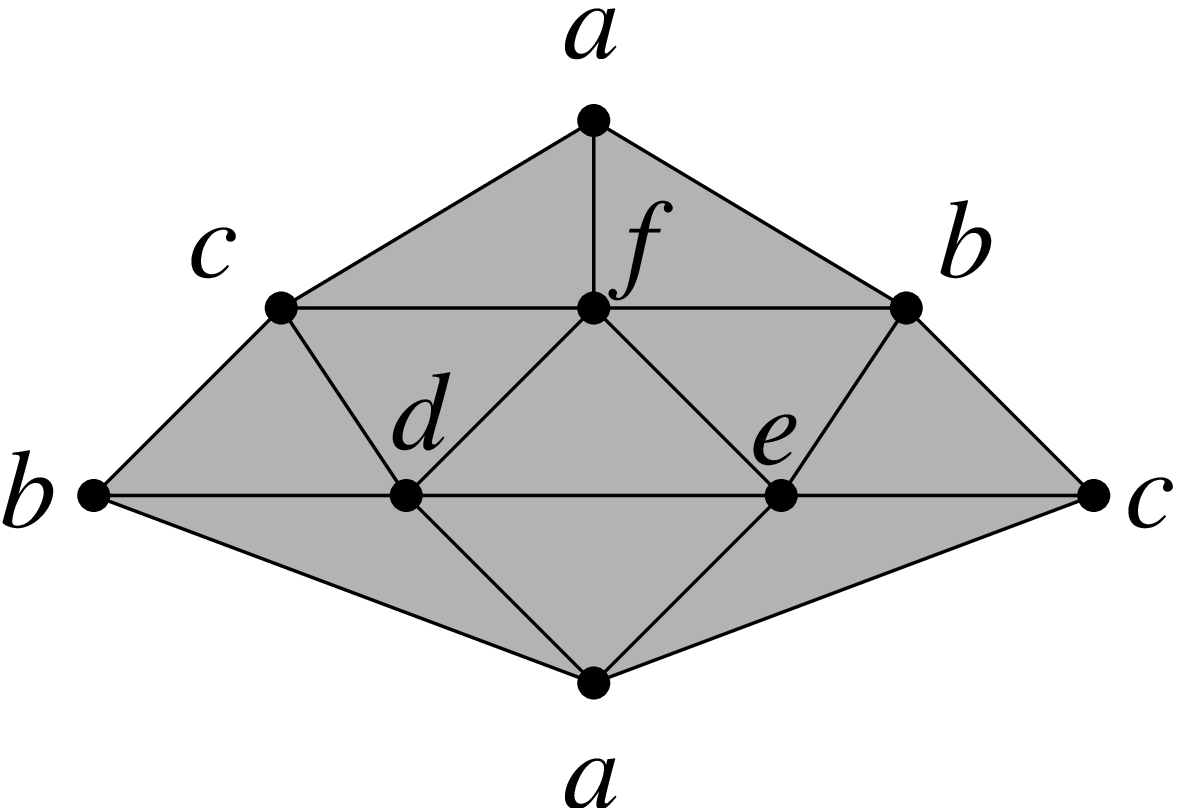}}
	\label{fig:real_proj_plane}} \\
\subfloat[A triangulation of the sphere]{\makebox[7cm]{
	\includegraphics[height=.9in]{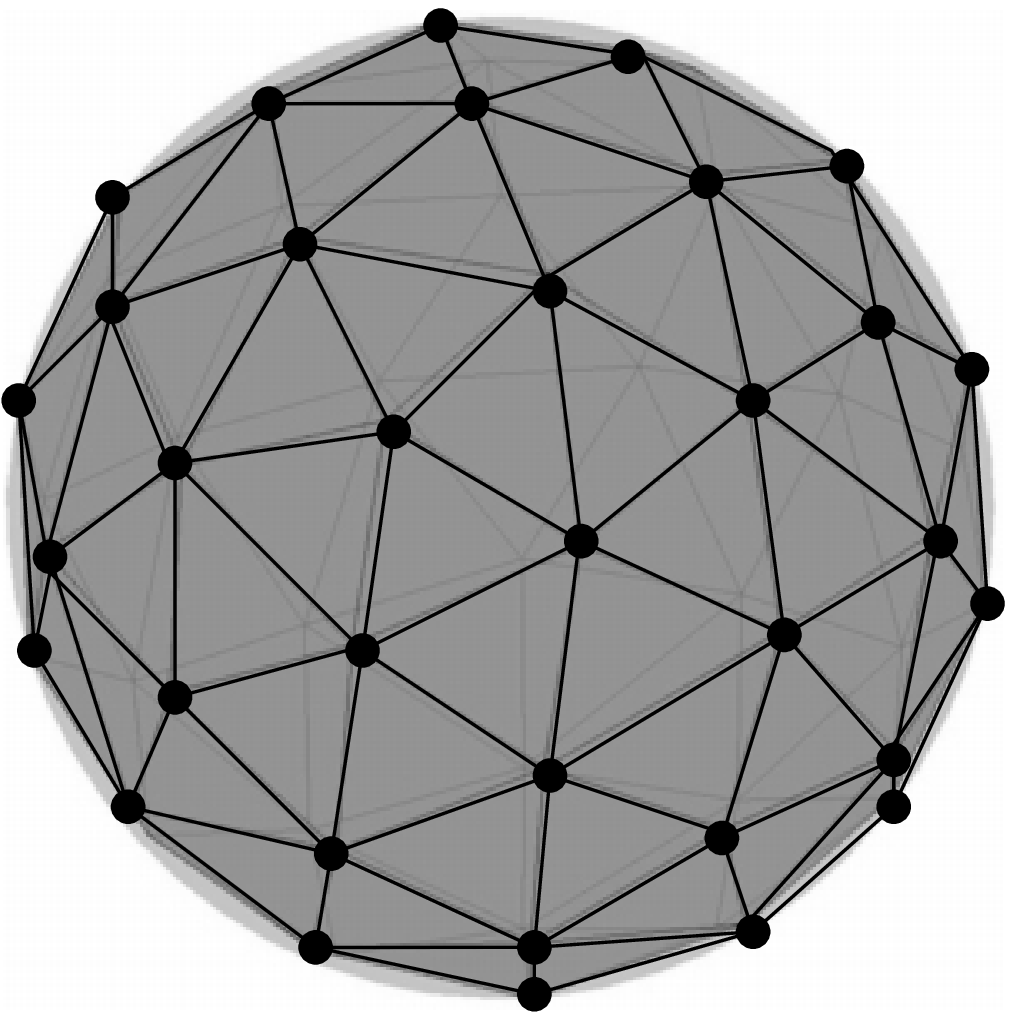}}
	\label{fig:sphere}} \qquad
\subfloat[A triangulation of the torus]{\makebox[7cm]{
	\includegraphics[height=.9in]{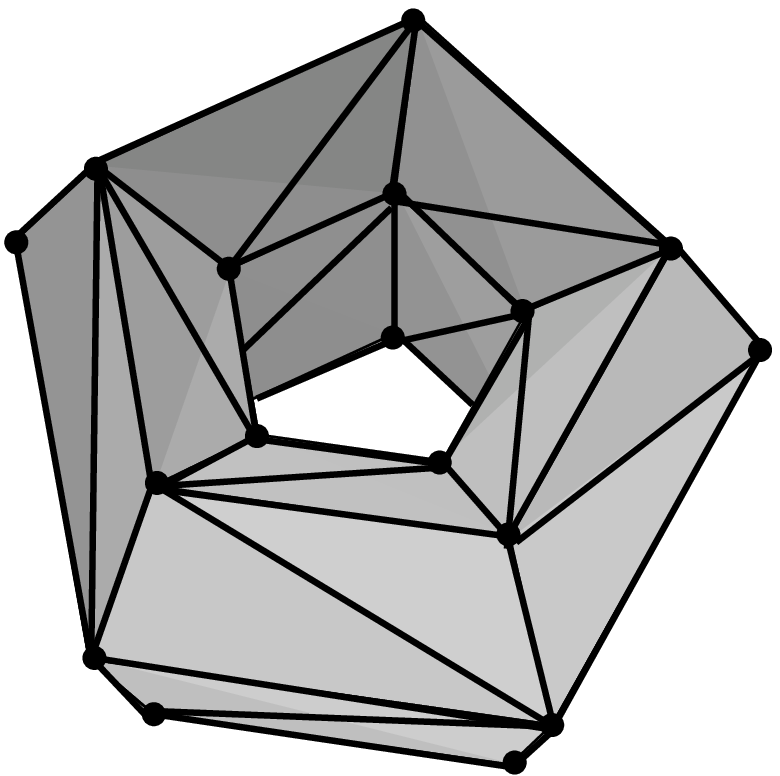}}
	\label{fig:torus}} \\
\subfloat[Two hollow tetrahedra glued along a $1$-face]{\makebox[7cm]{
	\includegraphics[height=.9in]{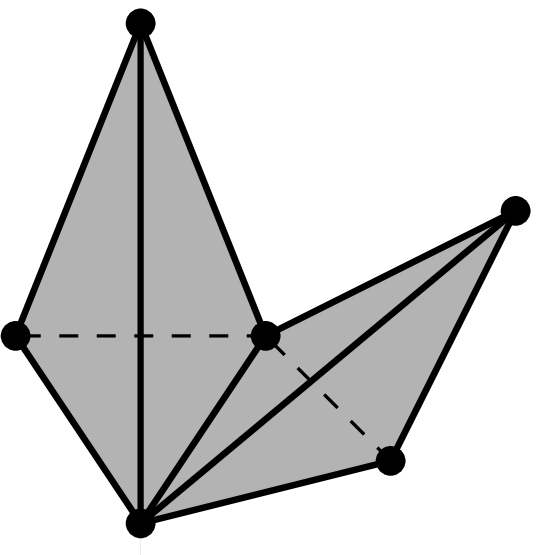}}
	\label{fig:two_tetra}}\qquad
\subfloat[A triangulation of the sphere pinched along a $1$-face]{\makebox[7cm]{
	\includegraphics[height=1in]{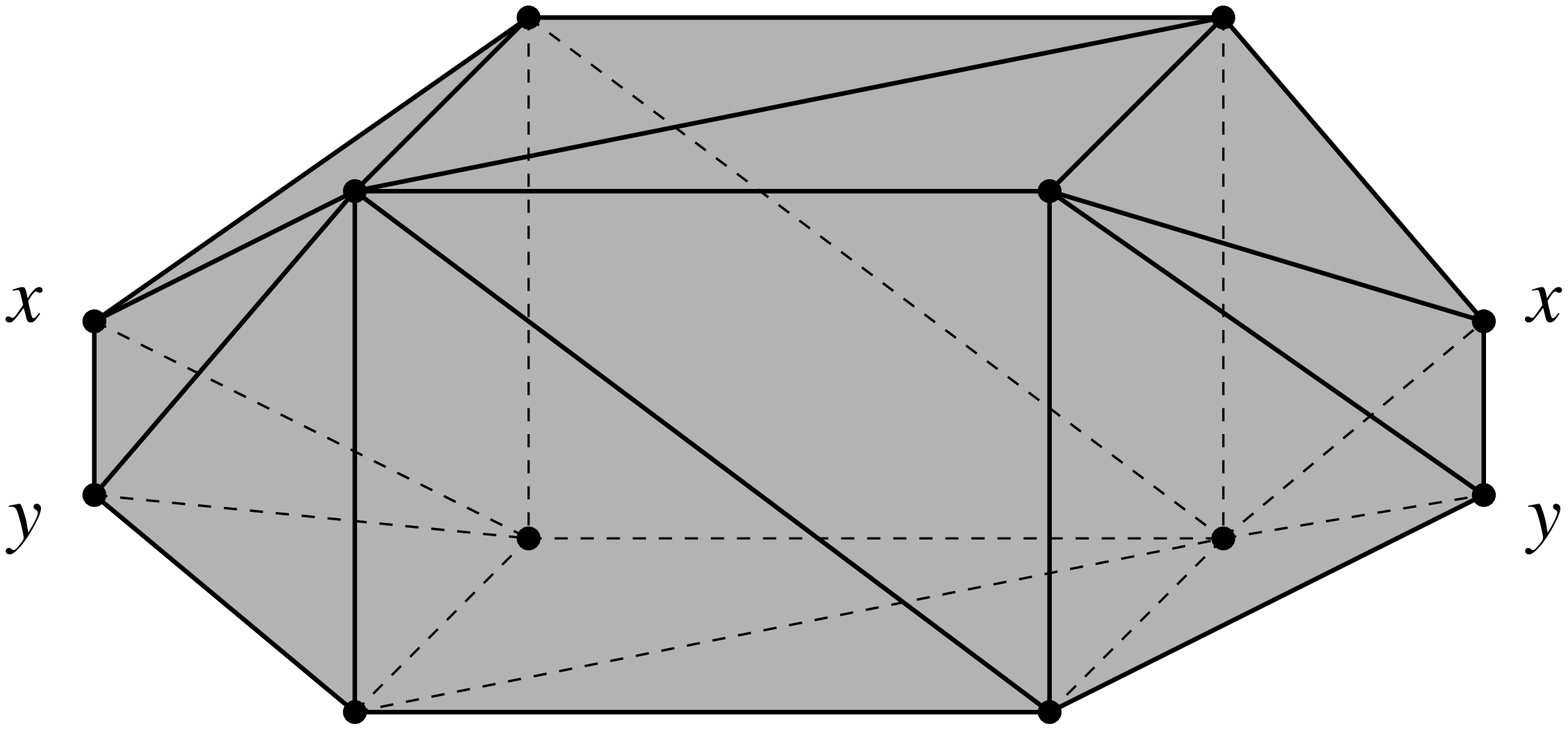}}
	\label{fig:pinched_sphere}}
\caption{Examples of $2$-dimensional cycles.} \label{fig:2d_cycles}
\end{figure}

Notice that a $d$-dimensional cycle has only one $d$-path-connected component.  Note also that the notion of a $d$-dimensional cycle is distinct from other versions of higher-dimensional cycles such as the {\bf Berge cycle} \cite{Berge89} and the {\bf simplicial cycle} \cite{CabFar11}, \cite{CFS11}.

The idea of a $d$-dimensional cycle is similar to the concept of a pseudo-manifold from algebraic topology (see for example~\cite{Munk84}). A {\bf pseudo $d$-manifold} is a pure $d$-dimensional $d$-path-connected simplicial complex in which every $(d-1)$-face is contained in exactly two $d$-faces.

\begin{definition}[{\bf orientable $d$-dimensional cycle}] \label{def:orient_ddim_cycle}
We say that a $d$-dimensional cycle $\Omega$ with $d$-faces $F_1,\ldots,F_k$ is {\bf orientable} if the following condition holds.
There exists a choice of orientations of $F_1,\ldots,F_k$ such that for any $(d-1)$-face $f$ of $\Omega$ when we consider the induced orientations of $f$ by the $F_i$'s containing $f$, these induced orientations are divided equally between the two orientation classes.  Otherwise we say that $\Omega$ is {\bf non-orientable}.
\end{definition}

Note that when we refer to the oriented $d$-faces of an orientable $d$-dimensional cycle we mean any set of orientations that is compatible with Definition \ref{def:orient_ddim_cycle}.

The $2$-dimensional cycles given in Figure \ref{fig:2d_cycles} are all orientable except for \ref{fig:real_proj_plane}, the triangulation of the real projective plane, which is non-orientable.


\subsection{Structure of $d$-dimensional cycles} \label{sec:struc_ddim_cycles}

\begin{definition}[{\bf face-minimal}]
A $d$-dimensional cycle $\Omega$ is called {\bf face-minimal} if there is no $d$-dimensional cycle on a strict subset of the $d$-faces of $\Omega$.
\end{definition}

It is not hard to see that a $1$-dimensional cycle is a graph cycle if and only if it is face-minimal.

\begin{definition}[{\bf vertex-minimal}]
A $d$-dimensional cycle $\Omega$ in a simplicial complex $\Gamma$ is called {\bf vertex-minimal} if there is no $d$-dimensional cycle in $\Gamma$ on a strict subset of the vertices of $\Omega$.
\end{definition}

The $2$-dimensional cycles in Figure \ref{fig:2d_cycles} are all face-minimal and vertex-minimal except for \ref{fig:two_tetra}.  This $2$-dimensional cycle is neither face-minimal nor vertex-minimal because it contains two $2$-dimensional cycles on strict subsets of $2$-faces and vertices each of which is a hollow tetrahedron.  The first author shows in~\cite{Con13} that a pseudo $d$-manifold is a face-minimal $d$-dimensional cycle.  The converse does not hold as we can see from Figure \ref{fig:pinched_sphere} since in that simplicial complex the face $\{x,y\}$ belongs to four distinct $2$-faces.

\begin{remark}
Notice that the face-minimality of a $d$-dimensional cycle $\Omega$ is not affected by whether or not it sits inside a larger simplicial complex.  On the other hand a $d$-dimensional cycle can be vertex-minimal when considered as a stand-alone simplicial complex, but not vertex-minimal when considered as a subcomplex of another simplicial complex.  As an example consider the simplicial complex in Figure \ref{fig:sphere_tetra}.  The outer sphere is not vertex-minimal in this complex as there exists another $2$-dimensional cycle on a strict subset of its vertices.  However, when considered as a simplicial complex on its own, the outer sphere is a vertex-minimal $2$-dimensional cycle.
\end{remark}

\begin{figure}[h!]
{\centering
    \includegraphics[height=1.2in]{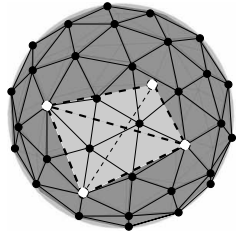}
    \caption {Triangulated sphere with suspended tetrahedron.} \label{fig:sphere_tetra}

}
\end{figure}

We can specialize the concepts of face-minimality and vertex-minimality to the case of orientable cycles in the following way.

\begin{definition}[{\bf orientably-face-minimal}]
An orientable $d$-dimensional cycle is called {\bf orientably-face-minimal} if there is no orientable $d$-dimensional cycle on a strict subset of its $d$-faces.
\end{definition}

\begin{definition}[{\bf orientably-vertex-minimal}]
An orientable $d$-dimensional cycle in a simplicial complex $\Gamma$ is called {\bf orientably-vertex-minimal} if there is no orientable $d$-dimensional cycle in $\Gamma$ on a strict subset of its vertices.
\end{definition}

It is easy to see that an orientable $d$-dimensional cycle can be orientably-vertex-minimal without being vertex-minimal.  It is not clear however, whether an orientable $d$-dimensional cycle can be orientably-face-minimal without being face-minimal.

The following lemma from \cite{Con13} demonstrates that every $d$-dimensional cycle can be decomposed into cycles which are face-minimal.

\begin{lem}[Connon~\cite{Con13} {\rm{\bf A $d$-dimensional cycle can be broken into face-minimal ones}}] \label{lem:cycle_decomp}
Any $d$-dimensional cycle $\Omega$ can be partitioned into face-minimal $d$-dimensional cycles $\Phi_1,\ldots,\Phi_n$.  In other words every $d$-face of $\Omega$ belongs to some $\Phi_i$ and no two distinct cycles $\Phi_i$ and $\Phi_j$ share a $d$-face.
\end{lem}

Of importance to us is the $d$-dimensional $d$-complete complex on $n$ vertices.  We denote this complex by $\Lambda_n^d$.

\begin{example}
The hollow tetrahedron $\Lambda_4^2$ is shown in Figure \ref{fig:hollow_tetra}.  It is the boundary of a $3$-simplex and as we see in the following lemma it is the $2$-dimensional cycle on the smallest number of vertices.
\end{example}

\begin{prop} [Connon~\cite{Con13} {\rm{\bf The smallest $d$-dimensional cycle is a complete one}}] \label{prop:smallest_cycle}
The smallest number of vertices that a $d$-dimensional cycle can have is $d+2$ and the only $d$-dimensional cycle on $d+2$ vertices is $\Lambda_{d+2}^{d}$.  In addition, $\Lambda_{d+2}^d$ is orientable.
\end{prop}


\section{$d$-chorded and $d$-cycle-complete complexes} \label{sec:dchorded_complexes}

In a chordal graph any cycle can be ``broken down'' into a set of complete cycles, or triangles, on the same vertex set, often in more than one way.  The clique complex of a chordal graph from Section \ref{sec:F_orig_thm} ``fills in'' these cycles by turning each such triangle into a face of the complex.  The original cycle can be thought of as a kind of sum of these triangles.  In this way, all $1$-dimensional homology existing in the chordal graph when it is thought of as a simplicial complex disappears in the clique complex, as all $1$-cycles are transformed into $1$-boundaries.

We would like to replicate this dismantling of a cycle into smaller
complete pieces in higher dimensions.  This is the motivation behind
the ideas of a {\bf $d$-chorded simplicial complex} and a {\bf $d$-cycle-complete-complex}.

In the case of chordal graphs, cycles that are not complete must have
a chord.  This chord breaks the cycle into two cycles both having
fewer vertices than the original.  It is the inductive nature of this
addition of chords which results in each cycle being dismantled into
complete cycles.  To achieve this same goal in $d$-dimensional cycles
we introduce the higher-dimensional notion of a {\bf chord set}.

\begin{definition}[{\bf chord set}]
Let $\Omega$ be a $d$-dimensional cycle in a simplicial complex $\Gamma$.  A {\bf chord set} of $\Omega$ in $\Gamma$ is a set $C$ of $d$-faces of $\Gamma\setminus \Omega$ contained in $V(\Omega)$ such that the simplicial complex $\langle C , \facets(\Omega) \rangle$ consists of $k$ $d$-dimensional cycles, $\Omega_1,\ldots,\Omega_k$, where $k \geq 2$ with the following conditions:
\begin{enumerate}
\item $\bigcup_{i=1}^k \facets(\Omega_i) = \facets(\Omega) \cup C$,
\item each $d$-face in $C$ is contained in an even number of the cycles $\Omega_1,\ldots,\Omega_k$,
\item each $d$-face of $\Omega$ is contained in an odd number of the cycles $\Omega_1,\ldots,\Omega_k$,
\item $|V(\Omega_i)|<|V(\Omega)|$ for $i=1,\ldots,k$.
\end{enumerate}
\end{definition}

A chord set of a graph cycle corresponds to a set of chords of the cycle in the traditional sense.  A chord of a graph cycle is always a chord set.  In Figure \ref{fig:graph_chord_set} we have a graph cycle on six vertices with a chord set displayed with dotted lines.  This cycle is broken into four smaller cycles by its chord set.

\begin{figure}[h!]
{\centering
    \includegraphics[height=.7in]{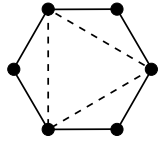}
    \caption {Example of a graph cycle with a chord set.} \label{fig:graph_chord_set}
}
\end{figure}

\begin{definition}[{\bf $d$-chorded}]
A pure $d$-dimensional simplicial complex $\Gamma$ is {\bf $d$-chorded} if all face-minimal $d$-dimensional cycles in $\Gamma$ that are not $d$-complete have a chord set in $\Gamma$.
\end{definition}

In the 1-dimensional case this definition says that a graph is $1$-chorded when all face-minimal cycles that are not $1$-complete have a chord set.  In other words, a graph is $1$-chorded when all graph cycles that are not triangles have a chord set.  This agrees with the usual notion of a chordal graph.

In Figure \ref{fig:2_chorded} we have examples of simplicial complexes that are $2$-chorded.  The hollow tetrahedron in Figure \ref{fig:hollow_tetra2} is $2$-chorded as it is a $2$-complete $2$-dimensional cycle.  The pure $2$-dimensional complex in Figure \ref{fig:chorded} is $2$-chorded because it contains no $2$-dimensional cycles.  The complexes in Figures \ref{fig:2tetra_chord} and \ref{fig:octahedron_chord} are $2$-chorded because they are face-minimal $2$-dimensional cycles with chord sets breaking the complexes into $2$-complete $2$-dimensional cycles.  The chord sets are shown in a darker colour.

\begin{figure}[h]
\centering
\subfloat[$2$-complete face-minimal $2$-dimensional cycle]{\makebox[7cm]{
            \includegraphics[height=1in]{hollow_tetra}}
            \label{fig:hollow_tetra2}} \qquad
\subfloat[Pure $2$-dimensional complex with no $2$-dimensional cycles]{\makebox[7cm]{
	\includegraphics[height=1in]{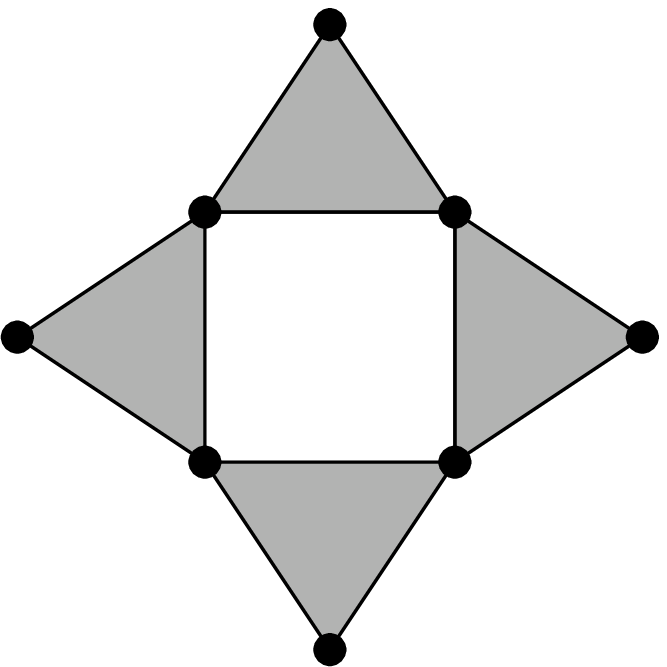}}
	\label{fig:chorded}} \\
\subfloat[Minimal $2$-dimensional cycle with chord set of size $1$]{\makebox[7cm]{
	\includegraphics[height=.65in]{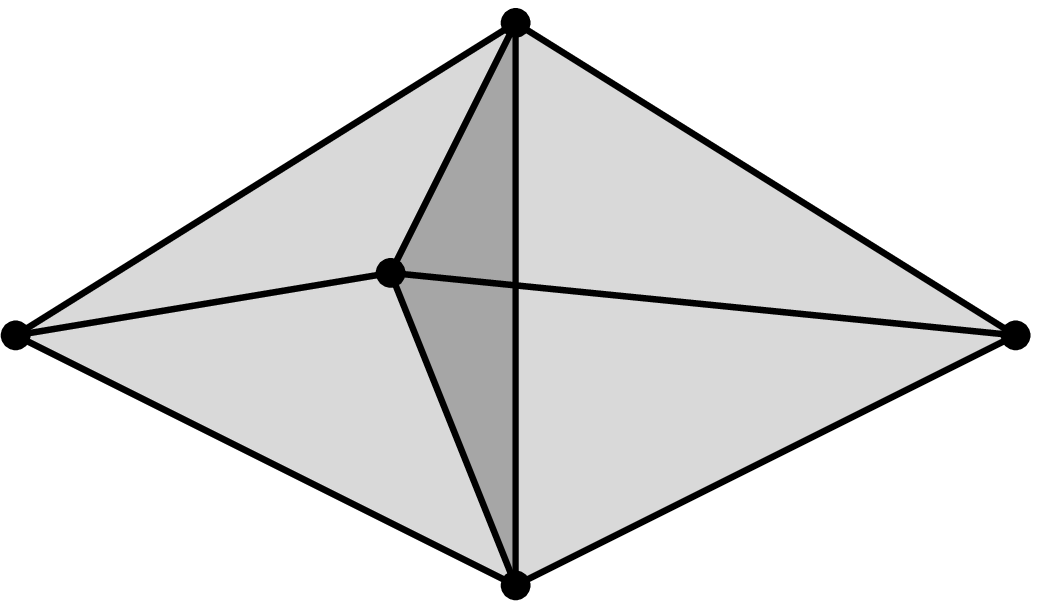}}
	\label{fig:2tetra_chord}} \qquad
\subfloat[Octahedron with chord set of size $4$]{\makebox[7cm]{
	\includegraphics[height=1.2in]{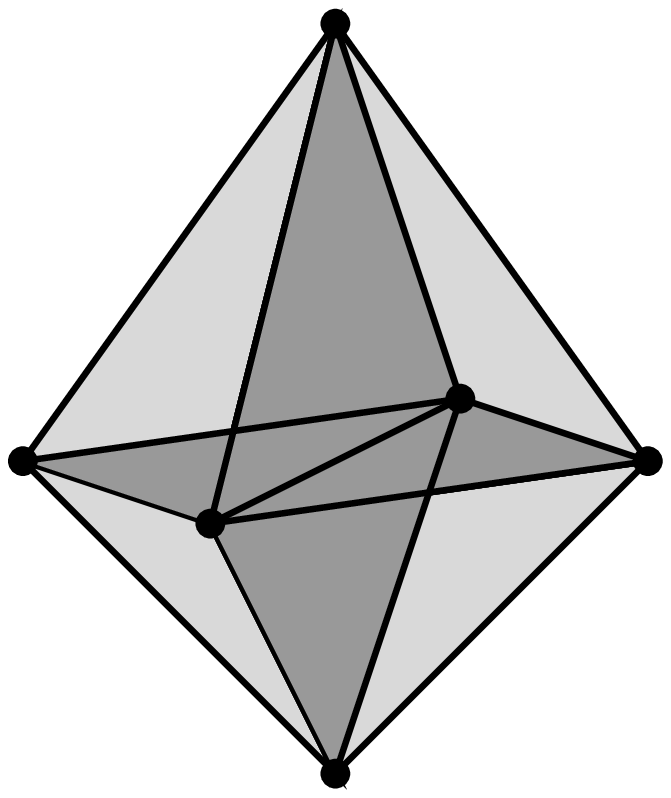}}
	\label{fig:octahedron_chord}}
\caption{Examples of $2$-chorded simplicial complexes.} \label{fig:2_chorded}
\end{figure}

\begin{remark} \label{rem:simplex_is_chorded}
It is not difficult to show that the pure $d$-skeleton of an $n$-simplex $\Gamma$ is $d$-chorded for any $d<n$.  The proof follows the same technique used in the proof of Theorem \ref{thm:one_dir} part 2 in Section \ref{sec:necessary_condition}.  As a starting point we need only notice that all induced subcomplexes of $\Gamma$ are also simplices and thus have reduced homology equal to zero in all dimensions.
\end{remark}

We will see in Section \ref{sec:necessary_condition} that our class of $d$-chorded complexes strictly contains the class of $(d+1)$-uniform {\bf chordal clutters} introduced by Woodroofe in \cite{Wood11} and the class of $(d+1)$-uniform {\bf generalized chordal hypergraphs} introduced by Emtander in \cite{Em10}.

We would like to make use of the homological characterization given in Theorem \ref{thm:Frob_hom_char} to extend Fr\"oberg's theorem to higher dimensions.  For this purpose, we require the property of being $d$-chorded to be transferred to induced subcomplexes.

\begin{lem} \label{lem:induced_complexes_chorded}
The pure $d$-skeleton of any induced subcomplex of a $d$-chorded simplicial complex is $d$-chorded.
\end{lem}

\begin{proof}
Let $\Gamma$ be a $d$-chorded simplicial complex and let $W \subseteq V(\Gamma)$.  Let $\Omega$ be any face-minimal $d$-dimensional cycle in $(\Gamma_W)^{[d]}$ that is not $d$-complete.  It is clear that $\Omega$ is also a $d$-dimensional cycle in $\Gamma$.  Also, $\Omega$ must be face-minimal in $\Gamma$ otherwise some strict subset of its $d$-faces is a $d$-dimensional cycle in $\Gamma$, but then also in $(\Gamma_W)^{[d]}$, which is a contradiction.  Hence since $\Gamma$ is $d$-chorded $\Omega$ has a chord set in $\Gamma$.  Since a chord set is a set of $d$-faces that lie on the same vertex set as the cycle, this chord set exists in $\Gamma_W$ as well.  Hence $\Omega$ has a chord set in $(\Gamma_W)^{[d]}$.  Thus every face-minimal $d$-dimensional cycle in $(\Gamma_W)^{[d]}$ that is not $d$-complete has a chord set.  Therefore $(\Gamma_W)^{[d]}$ is $d$-chorded.
\end{proof}

A chordal graph can also be defined without the notion of chords by requiring that all of its ``minimal'' cycles be complete.  We can extend this definition to higher dimensions in the following way.

\begin{definition}[{\bf $d$-cycle-complete, orientably-$d$-cycle-complete}]
A pure $d$-dimensional simplicial complex $\Gamma$ is called ({\bf orientably-}) {\bf $d$-cycle-complete} if all of its (orientably-) vertex-minimal $d$-dimensional cycles are $d$-complete.
\end{definition}

The examples in Figure \ref{fig:2_chorded} are all both $2$-cycle-complete and orientably-$2$-cycle-complete.  It is not hard to see that the set of $1$-cycle-complete and orientably-$1$-cycle-complete simplicial complexes corresponds exactly to the set of chordal graphs.

We have imposed structure on our classes of $d$-chorded and $d$-cycle-complete complexes by restricting the way in which higher-dimensional cycles may exist in these complexes.  A more severe restriction is to disallow these higher-dimensional cycles altogether on a particular level.

\begin{definition}[{\bf $d$-dimensional tree}]
A {\bf $d$-dimensional tree} is a pure $d$-dimensional simplicial complex with no $d$-dimensional cycles.
\end{definition}

Notice that the notion of a graph tree agrees with that of a $1$-dimensional tree.  It is also trivial to see that all $d$-dimensional trees are $d$-chorded.

Another higher-dimensional analogue to the graph tree is the {\bf simplicial tree} which is a connected simplicial complex with no simplicial cycles (see \cite{CabFar11}, \cite{Far02}).  It is not difficult to show that a pure $d$-dimensional simplicial tree is a $d$-dimensional tree as one can easily show that any $d$-dimensional cycle contains a simplicial cycle.

\begin{prop} [{\bf $d$-chorded $\Rightarrow$ $d$-cycle-complete $\Rightarrow$ orientably-$d$-cycle-complete}] \label{prop:classes_nested} \mbox{}
\begin{enumerate}
\item Any $d$-chorded simplicial complex is $d$-cycle-complete.
\item Any $d$-cycle-complete simplicial complex is orientably-$d$-cycle-complete.
\end{enumerate}
\end{prop}

\begin{proof}
~\begin{enumerate}
\item Let $\Gamma$ be a $d$-chorded simplicial complex and let $\Omega$ be any vertex-minimal $d$-dimensional cycle in $\Gamma$.  Suppose that $\Omega$ is not $d$-complete.  If $\Omega$ is face-minimal then $\Gamma$ contains a chord set for $\Omega$ which means that there exist $d$-dimensional cycles on strict subsets of the vertices of $\Omega$.  If $\Omega$ is not face-minimal then it contains a face-minimal $d$-dimensional cycle on its $d$-faces which has a chord set.  This also implies that there exist $d$-dimensional cycles on strict subsets of the vertices of $\Omega$.  Either way we have a contradiction to vertex-minimality of $\Omega$ and so $\Omega$ must be $d$-complete.  Hence $\Gamma$ is $d$-cycle-complete.

\item Let $\Gamma$ be a $d$-cycle-complete simplicial complex and let $\Omega$ be any orientably-vertex-minimal $d$-dimensional cycle in $\Gamma$.  We know that $\Omega$ does not contain any orientable $d$-dimensional cycles on a strict subset of its vertices.  If it also does not contain any non-orientable $d$-dimensional cycles on a strict subset of its vertices then it is vertex-minimal and so $d$-complete since $\Gamma$ is $d$-cycle complete.  Thus suppose that $\Omega$ contains a non-orientable cycle on a strict subset of its vertices.  If this non-orientable cycle is vertex-minimal then it is $d$-complete which means that it contains a copy of $\Lambda_{d+2}^d$, an orientable $d$-dimensional cycle on a strict subset of the vertices of $\Omega$.  This is a contradiction to the fact that $\Omega$ is orientably-vertex-minimal.  If the non-orientable cycle is not vertex-minimal then it must contain a $d$-dimensional cycle on a strict subset of its vertices which is vertex-minimal.  This cycle is $d$-complete since $\Gamma$ is $d$-cycle-complete and so contains a copy of $\Lambda_{d+2}^d$.  As before we have a contradiction and so $\Omega$ does not contain any non-orientable $d$-dimensional cycles on its vertex set.  Hence $\Omega$ is vertex-minimal and so $d$-complete since $\Gamma$ is $d$-cycle-complete.  Thus $\Gamma$ is orientably-$d$-cycle-complete.
\end{enumerate}
\end{proof}

By Proposition \ref{prop:classes_nested} we can see that the four classes of simplicial complexes defined in this section are ``nested'' with
\[
\begin{matrix}
    \{\mbox{$d$-dimensional trees}\} \cr
    \begin{rotate}{270}
        \Large $\subsetneq$
    \end{rotate}
    \cr \cr
    \{\mbox{$d$-chorded simplicial complexes}\} \cr
    \begin{rotate}{270}
        \Large $\subsetneq$
    \end{rotate}
    \cr \cr
    \{\mbox{$d$-cycle-complete simplicial complexes}\}  \cr
    \begin{rotate}{270}
     \Large  $\subsetneq$
    \end{rotate} \cr \cr
    \{\mbox{orientably-$d$-cycle-complete simplicial complexes}\}. \cr
\end{matrix}
\]

These inclusions are strict.  The simplicial complex $\Lambda_{d+2}^d$ is $d$-chorded but not a $d$-dimensional tree.  An example of a $2$-cycle-complete complex which is not $2$-chorded is given in Figure \ref{fig:sphere_tetra} in Section \ref{sec:struc_ddim_cycles}.  This complex is not $2$-chorded because it contains a face-minimal $2$-dimensional cycle, the triangulated sphere, which is not $2$-complete and has no chord set.  The triangulation of the real projective plane given in Figure \ref{fig:real_proj_plane} is an example of an orientably-$2$-cycle-complete complex which is not $2$-cycle-complete.  It is orientably-$2$-cycle-complete since it contains no orientable $2$-dimensional cycles, but it is not $2$-cycle-complete because it contains a vertex-minimal $2$-dimensional cycle that is not $2$-complete.


\section{Simplicial homology of $d$-dimensional cycles and related structures}\label{sec:homology_ddim_cycles}

As demonstrated by the first author in \cite{Con13}, the presence of a $d$-dimensional cycle in a simplicial complex has implications for the simplicial homology of that complex.  We see in the next proposition that, over $\Z_2$, $d$-dimensional cycles naturally arise as the support complexes of homological $d$-cycles.

\begin{prop} [Connon~\cite{Con13} {\rm{\bf $d$-dimensional cycles are $d$-cycles and conversely}}] \label{prop:ddimcycle_is_dcycle}
The sum of the $d$-faces of a $d$-dimensional cycle is a homological $d$-cycle over $\Z_2$ and, conversely, the $d$-path-connected components of the support complex of a homological $d$-cycle are $d$-dimensional cycles.
\end{prop}

The following theorem shows that over any field of characteristic $2$ the $d$-dimensional cycle is exactly the right notion to capture the property of non-zero homology.

\begin{thm} [Connon~\cite{Con13}] \label{thm:cycle_hom_z2}
For any simplicial complex $\Gamma$ and any field $k$ of characteristic $2$, $\tilde{H}_d(\Gamma;k) \neq 0$ if and only if $\Gamma$ contains a $d$-dimensional cycle, the sum of whose $d$-faces is not a $d$-boundary.
\end{thm}

In order to extend these results to an arbitrary field one must take
additional combinatorial characteristics into consideration.  In
\cite{Con13} the first author proves that orientable $d$-dimensional cycles
have non-zero homology over any field.

\begin{thm} [Connon~\cite{Con13} {\rm{\bf orientable $d$-dimensional cycles result in non-zero homology}}] \label{thm:orientable_cycle_hom}
For any simplicial complex $\Gamma$ and any field $k$, if $\Gamma$ contains an orientable $d$-dimensional cycle the sum of whose oriented $d$-faces is not a $d$-boundary then $\tilde{H}_d(\Gamma;k) \neq 0$.
\end{thm}

Using these results we can come closer to obtaining a class of complexes with the right homological conditions to satisfy Theorem \ref{thm:Frob_hom_char}.  It is this goal that motivated the introduction of $d$-chorded simplicial complexes and $d$-cycle-complete complexes in Section \ref{sec:dchorded_complexes}.

As mentioned in Section \ref{sec:dchorded_complexes}, the clique complex of a chordal graph removes all $1$-dimensional homology from its cycles by turning these cycles into sums of $2$-faces.  The idea of ``filling in'' complete subgraphs can easily be extended to simplicial complexes.

\begin{definition}[{\bf $d$-closure}]
Let $\Gamma$ be a pure $d$-dimensional simplicial complex with vertex set $V$.  We define $\Delta_d(\Gamma)$ to be the simplicial complex with vertex set $V$ and such that
\begin{enumerate}
\item $\Gamma \subseteq \Delta_d(\Gamma)$,
\item for all $S \subseteq V$ with $|S|\leq d$, we have $S \in \Delta_d(\Gamma)$, and
\item for any $S\subseteq V$ with $|S| > d+1$, if all $(d+1)$-subsets of $S$ are faces of $\Gamma$ then $S$ is a face of $\Delta_d(\Gamma)$.
\end{enumerate}
\end{definition}

In \cite{Em10} Emtander refers to $\Delta_d(\Gamma)$ as the {\bf complex of $\Gamma$} and in \cite{MYZ12} it is called the {\bf clique complex of $\Gamma$}.  We will refer to $\Delta_d(\Gamma)$ as the {\bf $d$-closure} or simply the {\bf closure} of $\Gamma$.  Note that when $G$ is a graph $\Delta_1(G)$ is equivalent to $\Delta(G)$, the clique complex of $G$.

In Figure \ref{fig:2closure} we give an example of a pure $2$-dimensional complex $\Gamma$ and its $2$-closure $\Delta_2(\Gamma)$.

\begin{figure}[h]
\centering
\subfloat[$\Gamma = \langle abc, abd, acd, bcd, cde \rangle$]{\makebox[6cm]{
            \includegraphics[height=1.3in]{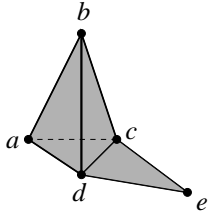}}
            \label{fig:2closure1}} \qquad \qquad
\subfloat[$\Delta_2(\Gamma) = \langle abcd, cde, ae, be \rangle$]{\makebox[6cm]{
	\includegraphics[height=1.3in]{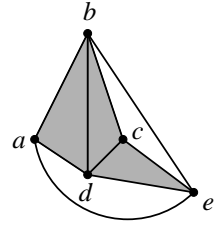}}
	\label{fig:2closure2}}
\caption{$2$-closure} \label{fig:2closure}
\end{figure}

It turns out that the closure operation commutes with the operation of taking induced subcomplexes.

\begin{lem} [{\bf Closure commutes with taking induced subcomplexes}] \label{lem:delta_operation_induced}
Let $\Gamma$ be a pure $d$-dimensional simplicial complex and let $W \subseteq V(\Gamma)$.  Then we have $\Delta_d(\Gamma)_W=\Delta_d((\Gamma_W)^{[d]})$.
\end{lem}

\begin{proof}
First note that all possible faces of dimension less than $d$ contained in $W$ exist in both complexes, by the nature of $d$-closure.  No other faces of dimension less than $d$ are possible as the vertex set of both complexes is $W$.  Next consider faces of dimension $d$.  Any face of dimension $d$ in $\Delta_d(\Gamma)_W$ is a face of $\Gamma$ and is contained in $W$.  Such a face is clearly a face of $(\Gamma_W)^{[d]}$ and so of $\Delta_d((\Gamma_W)^{[d]})$.  Similarly any face of dimension $d$ in $\Delta_d((\Gamma_W)^{[d]})$ is a face of $\Gamma_W$ so is a face of $\Gamma$ and lies in $W$.  So it is a face of $\Delta_d(\Gamma)$ and $\Delta_d(\Gamma)_W$ in particular.  Next consider a face $F$ of dimension greater than $d$ that lies in $\Delta_d(\Gamma)_W$.  Such a face lies in $W$ and, by the nature of $d$-closure, all possible subsets of the face of size $d+1$ are $d$-faces of $\Gamma$.  Since $F \subseteq W$, these $d$-faces are also faces of $(\Gamma_W)^{[d]}$ and so $F$ lies in $\Delta_d((\Gamma_W)^{[d]})$.  If $F$ is a face of dimension larger than $d$ in the complex $\Delta_d((\Gamma_W)^{[d]})$ then $F$ lies in $W$ and all possible subsets of the face of size $d+1$ are $d$-faces of $(\Gamma_W)^{[d]}$.  Since these $d$-faces must be faces of $\Gamma$, $F$ lies in $\Delta_d(\Gamma)_W$.  Therefore we have $\Delta_d(\Gamma)_W= \Delta_d(\Gamma_W)$.
\end{proof}

In the second half of Theorem \ref{thm:Frob_original}, Fr\"oberg states that if the Stanley-Reisner ideal of a simplicial complex has a $2$-linear resolution then the complex is equal to the clique complex of its $1$-skeleton.  We can easily extend this idea to the higher-dimensional closure operation.  In fact we only require that the generators of the ideal have the same degree.  Note that one direction of the following proposition is also given by \cite[Proposition 4.4]{MYZ12}.

\begin{prop}[{\bf Minimal generation in fixed degree}]\label{prop:min_gen}
The Stanley-Reisner ideal of a simplicial complex $\Gamma$ is minimally generated in degree $d+1$ if and only if $\Gamma = \Delta_d(\Gamma^{[d]})$.
\end{prop}

\begin{proof}
Suppose that $\N(\Gamma)$ is minimally generated in degree $d+1$. Let $F$ be any face of $\Delta_d(\Gamma^{[d]})$.  First suppose that $|F| < d+1$.  If $F$ is not a face of $\Gamma$ then $x^F \in \N(\Gamma)$.  However, $\N(\Gamma)$ is minimally generated by elements of degree $d+1$ and so we have a contradiction.  Therefore $F \in \Gamma$.

Suppose now that $|F|=d+1$.  By the definition of $d$-closure if $F \in \Delta_d(\Gamma^{[d]})$ then $F \in \Gamma^{[d]}$.  Hence we must have $F \in \Gamma$ also.

If $|F|>d+1$ then by the definition of $\Delta_d(\Gamma^{[d]})$ all $(d+1)$-subsets of $F$ are faces of $\Gamma^{[d]} \subseteq \Gamma$.  If $F$ is not a face of $\Gamma$ then we know that $x^F \in \N(\Gamma)$.  Hence $x^F$ is divisible by some monomial of degree $d+1$ whose elements make up a non-face of $\Gamma$.  This is not possible since all $(d+1)$-subsets of $F$ are faces of $\Gamma$. Therefore $F$ must be a face of $\Gamma$.  We conclude that $\Delta_d(\Gamma^{[d]}) \subseteq \Gamma$.

Now let $F$ be a face of $\Gamma$.  If $|F| < d+1$ then $F$ is automatically a face of $\Delta_d(\Gamma^{[d]})$.  If $|F| = d+1$ then $F$ is a face of $\Gamma^{[d]}$ and so a face of $\Delta_d(\Gamma^{[d]})$.  If $|F| > d+1$ then all $(d+1)$-subsets of $F$ are clearly faces of $\Gamma^{[d]}$.  By the definition of $d$-closure we have that $F$ is a face of $\Delta_d(\Gamma^{[d]})$.  Hence all faces of $\Gamma$ are faces of $\Delta_d(\Gamma^{[d]})$.  Therefore $\Gamma = \Delta_d(\Gamma^{[d]})$.

Now suppose that $\Gamma = \Delta_d(\Gamma^{[d]})$.  Then $\Gamma$ contains all possible faces of dimension less than $d$ by definition.  Also, any subset of vertices of size at least $d+2$, all of whose subsets are faces of $\Gamma$, must also be a face of $\Gamma$ by the definition of $d$-closure.  Hence all minimal non-faces of $\Gamma$ must have size exactly $d+1$.  Thus $\N(\Gamma)$ is minimally generated in degree $d+1$.
\end{proof}

By Theorem \ref{thm:Frob_hom_char}, in order to show that a pure $d$-dimensional simplicial complex has a linear resolution we must show that the homology of the $d$-closure of the complex and the homologies of the induced subcomplexes are zero in all dimensions except $d-1$.  By the nature of the closure operation it is trivial to see that, for the $d$-closure of a complex, all homology groups in dimension less than $d-1$ are zero as all faces of dimension less than $d$ are added by this operation.  We will show next that when the complex is $d$-chorded the $d$-level homology of the closure is also zero.

\begin{lem} \label{lem:dchorded_boundary}
Let $\Omega$ be a $d$-dimensional cycle in a $d$-chorded complex $\Gamma$.  The sum of the $d$-faces of $\Omega$ forms a $d$-boundary on $V(\Omega)$ in $\Delta_d(\Gamma)$ over $\Z_2$.
\end{lem}

\begin{proof}
We will use induction on the number of vertices of $\Omega$.  By Proposition \ref{prop:smallest_cycle}, the fewest number of vertices that a $d$-dimensional cycle can have is $d+2$ and this occurs when $\Omega = \Lambda_{d+2}^{d}$.  In this case $\Delta_d(\Gamma)_{V(\Omega)}$ is a $(d+1)$-simplex and so the sum of the faces of $\Omega$ forms the $d$-boundary of $\Delta_d(\Gamma)_{V(\Omega)}$ on $V(\Omega)$.

Now suppose that the statement holds for all $d$-dimensional cycles with fewer than $n$ vertices and let $\Omega$ have $n$ vertices.  If $\Omega$ is not face-minimal then by Lemma \ref{lem:cycle_decomp} it can be partitioned into face-minimal $d$-dimensional cycles.  To show that the sum of the $d$-faces of $\Omega$ forms a $d$-boundary in $\Delta_d(\Gamma)_{V(\Omega)}$ we need only show that the sum of the $d$-faces of each such face-minimal cycle forms a $d$-boundary in $\Delta_d(\Gamma)_{V(\Omega)}$ since then we may add them together to show that the original sum is a $d$-boundary.  Therefore, without loss of generality, we may assume that $\Omega$ is a face-minimal $d$-dimensional cycle.

If $\Omega$ is $d$-complete then $\Delta_d(\Gamma)_{V(\Omega)}$ is an $(n - 1)$-simplex in $\Delta_d(\Gamma)$ and so the sum of the $d$-faces of $\Omega$ forms a $d$-boundary on $V(\Omega)$.  If $\Omega$ is not complete then since $\Gamma$ is $d$-chorded there exists a chord set $C$ of $\Omega$ in $\Gamma$.  Let the $d$-dimensional cycles associated to $C$ be $\Omega_1,\ldots,\Omega_k$.  We know that $|V(\Omega_i)| < |V(\Omega)|$ for all $i$ and so, by induction, the sum of the $d$-faces of $\Omega_i$ forms a $d$-boundary in $\Delta_d(\Gamma)_{V(\Omega_i)}$ over $\Z_2$.  By properties 2 and 3 of a chord set, over $\Z_2$ we have
\[
    \sum_{i=1}^k \left(\sum (d\textrm{-faces of } \Omega_i)\right) = \sum (d\textrm{-faces of } \Omega).
\]
Therefore the sum of the $d$-faces of $\Omega$ is the $d$-boundary in $\Delta_d(\Gamma)_{V(\Omega)}$ of the sum of the $(d+1)$-faces for which the sums of the $d$-faces of the $\Omega_i$'s are $d$-boundaries.
\end{proof}

\begin{prop} [{\bf Vanishing homologies in $d$-closure of $d$-chorded complexes}] \label{prop:chorded_dhom}
Let $\Gamma$ be a $d$-chorded simplicial complex.  Then for any $W \subseteq V(\Gamma)$ and any field $k$ of characteristic $2$ we have $\tilde{H}_i(\Delta_d(\Gamma)_W;k)=0$ for $0 \leq i \leq d-2$ and $i=d$.
\end{prop}

\begin{proof}
Using the Universal Coefficient Theorem (see \cite[Theorem 3A.3]{Hatch02}) it is enough to show this for the case $k=\Z_2$.

From the discussion preceding Lemma \ref{lem:dchorded_boundary} and by Lemmas \ref{lem:induced_complexes_chorded} and \ref{lem:delta_operation_induced} we know that \newline $\tilde{H}_i(\Delta_d(\Gamma)_W;\Z_2)=0$ for $0 \leq i \leq d-2$ and any $W \subseteq V(\Gamma)$.  We need only show that $\tilde{H}_d(\Delta_d(\Gamma)_W;\Z_2)=0$ for any $W \subseteq V(\Gamma)$.  Again by Lemmas \ref{lem:induced_complexes_chorded} and \ref{lem:delta_operation_induced} it is enough to show that $\tilde{H}_d(\Delta_d(\Gamma);\Z_2)=0$.

If $\tilde{H}_d(\Delta_d(\Gamma);\Z_2)\neq 0$ then by Theorem \ref{thm:cycle_hom_z2} we know that $\Delta_d(\Gamma)$ contains a $d$-dimensional cycle, the sum of whose $d$-faces does not form a $d$-boundary.  This contradicts Lemma \ref{lem:dchorded_boundary} as a $d$-dimensional cycle in $\Delta_d(\Gamma)$ is a $d$-dimensional cycle in $\Gamma$ also.  Therefore we must have $\tilde{H}_d(\Delta_d(\Gamma);\Z_2)=0$.
\end{proof}

It will also be of use to us that, over $\Z_2$, certain $d$-dimensional cycles whose faces form the support complex of a $d$-boundary have chord sets.

\begin{lem} \label{lem:bdry_cycles_have_chord_sets}
Let $\Omega$ be a face-minimal $d$-dimensional cycle that is not $d$-complete in a simplicial complex $\Gamma$.  If, over $\Z_2$, $\Omega$ is the support complex of a $d$-boundary of faces of $\Gamma_{V(\Omega)}$ then $\Omega$ has a chord set in $\Gamma$.
\end{lem}

\begin{proof}
Let the $d$-faces of $\Omega$ be $F_1,\ldots,F_k$.   Since, in $\Z_2$, $\Omega$ forms the support complex of a $d$-boundary of faces of $\Gamma_{V(\Omega)}$, there exist $(d+1)$-faces $G_1,\ldots,G_\ell$ of $\Gamma_{V(\Omega)}$ such that
\[
    \partial_{d+1}\left(\sum_{i=1}^\ell G_i\right) = \sum_{i=1}^k F_i.
\]
Note that $\ell \geq 2$ since otherwise we have
\[
    \partial_{d+1}(G_1)= \sum_{i=1}^k F_i
\]
and since $\langle G_1\rangle ^{[d]} = \Lambda_{d+2}^{d}$ this indicates that $\Omega$ itself is $\Lambda_{d+2}^{d}$.  This can't happen since $\Omega$ is not $d$-complete.

Let $E_1,\ldots,E_m$ be the $d$-faces of $G_1,\ldots,G_\ell$ which are not $d$-faces of $\Omega$.  Note that, since $\langle G_i\rangle ^{[d]}=\Lambda_{d+2}^{d}$ with $G_i \subseteq V(\Omega)$ for all $i$ and $\Omega$ is a face-minimal non-$d$-complete $d$-dimensional cycle, the set $\{E_1,\ldots,E_m\}$ is non-empty.  We claim that $\{E_1,\ldots,E_m\}$ is a chord set of $\Omega$ in $\Gamma$.

First note that the vertices of $E_1,\ldots,E_m$ are contained in $V(\Omega)$ since $G_1,\ldots,G_\ell$ are faces of $\Gamma_{V(\Omega)}$. Also, by construction we have $\{E_1,\ldots,E_m\} \cap \{F_1,\ldots,F_k\}=\emptyset$.  We set $\Omega_i = \langle G_i\rangle ^{[d]}$ for each $i \in \{1,\ldots,\ell\}$.

Since each $d$-face of the two sets $\{F_1,\ldots,F_k\}$ and $\{E_1,\ldots,E_m\}$ appears in at least one of the $G_i$'s, we have
\[
    \bigcup_{i=1}^\ell \facets(\Omega_i) = \facets(\Omega) \cup \{E_1,\ldots,E_m\}.
\]
Over $\Z_2$ we have
\[
\partial_{d+1}\left(\sum_{i=1}^\ell G_i\right) = \sum_{i=1}^k F_i
\]
and so we know that for each $1\leq i \leq m$ the face $E_i$ must be contained in an even number of the faces $G_1,\ldots,G_\ell$ as $E_i$ does not appear on the right-hand side of this equation.  Therefore $E_i$ is also contained in an even number of the cycles $\Omega_1,\ldots,\Omega_{\ell}$.  Similarly each $d$-face of $\Omega$ must be contained in an odd number of the cycles $\Omega_1,\ldots,\Omega_{\ell}$.

Finally since $\Omega$ is not $d$-complete, we know that $|V(\Omega)|> d+2$ by Proposition \ref{prop:smallest_cycle}.  Therefore since $|V(\Omega_i)|=d+2$ for all $i \in \{1,\ldots,\ell\}$, we have that $|V(\Omega_i)|<|V(\Omega)|$ for all $i \in \{1,\ldots,\ell\}$.  Hence $\{E_1,\ldots,E_m\}$ is a chord set of $\Omega$ in $\Gamma$.
\end{proof}


\section{A necessary condition for a monomial ideal to have a linear resolution} \label{sec:necessary_condition}

Whether or not a monomial ideal has a linear resolution over a field $k$ depends on the characteristic of $k$.  A typical example of this is demonstrated by the triangulation of the real projective plane shown in Figure \ref{fig:real_proj_plane}.  The Stanley-Reisner ideal of the $2$-closure of this simplicial complex has a linear resolution only when the characteristic of $k$ is not 2.  This complex is an example of a non-orientable $2$-dimensional cycle and it demonstrates non-zero homology in dimension $2$ only over fields with characteristic $2$.

It turns out however, that when a square-free monomial ideal has a linear resolution this forces restrictions on the orientable $d$-dimensional cycles of the associated simplicial complex regardless of the field in question.  In the case that the field has characteristic $2$ the resulting complexes are forced to be in an even smaller class.

\begin{thm}[{\bf Main theorem}]\label{thm:one_dir}
Let $\Gamma$ be a simplicial complex, let $k$ be any field and let $d \geq 1$.  If the Stanley-Reisner ideal of $\Gamma$ has a $(d+1)$-linear resolution over $k$ then $\Gamma = \Delta_d(\Gamma^{[d]})$ and
\begin{enumerate}
\item $\Gamma^{[d]}$ is orientably-$d$-cycle-complete
\item $\Gamma^{[d]}$ is $d$-chorded if $k$ has characteristic $2$.
\end{enumerate}
\end{thm}

\begin{proof}
Since $\N(\Gamma)$ has a $(d+1)$-linear resolution, it is minimally generated in degree $d+1$.  Therefore by Proposition \ref{prop:min_gen} we have $\Gamma = \Delta_d(\Gamma^{[d]})$.

\begin{enumerate}
\item For a contradiction, let $\Omega$ be any orientably-vertex-minimal $d$-dimensional cycle in $\Gamma^{[d]}$ which is not $d$-complete.  Let the $d$-faces of $\Omega$ be $F_1,\ldots,F_k$ and let $W=\bigcup_{i=1}^k F_i$.  Since $\Omega$ is not $d$-complete, $|W|>d+2$ by Proposition \ref{prop:smallest_cycle}.  We claim that $\Gamma_W=(\Delta_d(\Gamma^{[d]}))_W$ has dimension $d$.  To show this we must demonstrate that every $(d+2)$-subset of $W$ contains a $(d+1)$-subset which is not a face of $\Gamma^{[d]}$.

Suppose that there is some $(d+2)$-subset $S$ of $W$ such that all $(d+1)$-subsets of $S$ are faces of $\Gamma^{[d]}$. Then $\Gamma^{[d]}_S$ is $d$-complete and so by Proposition \ref{prop:smallest_cycle} $\Gamma^{[d]}_S$ is an orientable $d$-dimensional cycle. This is a contradiction since $\Omega$ is orientably-vertex-minimal and $d+2=|S|<|W|$.  Therefore every $(d+2)$-subset of $W$ must contain a $(d+1)$-subset which is not a face of $\Gamma^{[d]}$.  Therefore by the definition of $d$-closure $(\Delta_d(\Gamma^{[d]}))_W=\Gamma_W$ cannot contain any faces of size $d+2$ or higher.  Hence $\Gamma_W$ has dimension $d$.

Since $\dim \Gamma_W =d$, the sum of the $d$-faces of $\Omega$ cannot be a $d$-boundary.  Therefore by Theorem \ref{thm:orientable_cycle_hom}  we know that $\tilde{H}_d(\Gamma_W;k) \neq 0$.  This is a contradiction to Theorem \ref{thm:Frob_hom_char} since $\N(\Gamma)$ has a $(d+1)$-linear resolution. Therefore $\Gamma^{[d]}$ has no orientably-vertex-minimal $d$-dimensional cycles which are not $d$-complete.  Hence $\Gamma^{[d]}$ is orientably-$d$-cycle-complete.

\item Let $\Omega$ be any face-minimal $d$-dimensional cycle in $\Gamma^{[d]}$ which is not $d$-complete.   Since $\N(\Gamma)$ has a $(d+1)$-linear resolution over $k$ and $\Gamma = \Delta_d(\Gamma^{[d]})$, we know by Theorem \ref{thm:Frob_hom_char} that
\[
	\tilde{H}_d(\Delta_d(\Gamma^{[d]})_{V(\Omega)};k)=0.
\]
Using the Universal Coefficient Theorem (\cite[Theorem 3A.3]{Hatch02}) we find that
\[
        \tilde{H}_d(\Delta_d(\Gamma^{[d]})_{V(\Omega)};\Z_2)=0.
\]
Let the $d$-faces of $\Omega$ be $F_1,\ldots,F_k$.  By Proposition \ref{prop:ddimcycle_is_dcycle} $\sum_{i=1}^k F_i$ is a homological $d$-cycle over $\Z_2$.  Therefore $\sum_{i=1}^k F_i$ is a $d$-boundary in $\Delta_d(\Gamma^{[d]})_{V(\Omega)}$.  Hence by Lemma \ref{lem:bdry_cycles_have_chord_sets} we know that $\Omega$ has a chord set in $\Gamma^{[d]}$.  Therefore $\Gamma^{[d]}$ is $d$-chorded. \qedhere
\end{enumerate}
\end{proof}

Theorem \ref{thm:one_dir} includes a generalization of one direction of Fr\"oberg's theorem.  It is equivalent to Theorem \ref{thm:converse_intro} given in Section \ref{sec:intro}.  This follows from the fact that the facet ideal of the $d$-complement of a simplicial complex $\Gamma$ is the Stanley-Reisner ideal of the $d$-closure of $\Gamma$.

The following corollary to Theorem \ref{thm:one_dir} part 2 gives us a necessary condition for a square-free monomial ideal to have a linear resolution over all fields.

\begin{corollary} [{\bf Linear resolution over all fields implies $d$-chorded}] \label{cor:one_dir_over_all}
Let $I$ be a square-free monomial ideal with Stanley-Reisner complex $\Gamma$. If $I$ has a $(d+1)$-linear resolution over all fields then $\Gamma^{[d]}$ is $d$-chorded.
\end{corollary}

As a consequence of either Theorem \ref{thm:one_dir} or Corollary \ref{cor:one_dir_over_all} we see that the class of $d$-chorded complexes contains the class of $(d+1)$-uniform {\bf chordal clutters} introduced by Woodroofe in \cite{Wood11} and the class of $(d+1)$-uniform {\bf generalized chordal hypergraphs} introduced by Emtander in \cite{Em10} since the hypergraphs in these classes have complements whose edge ideals have linear resolutions over all fields. However, consider the complex in Figure \ref{fig:chorded_not_chordal} which consists of four hollow tetrahedra ``glued together''.  This is a $2$-chorded simplicial complex, which is not chordal in the sense of \cite{Em10} or \cite{Wood11} when considered as a hypergraph, and yet the Stanley-Reisner ideal of the $2$-closure of this complex has a $3$-linear resolution over all fields.  This answers, in the positive, a question posed by Emtander in Section 5 of \cite{Em10} which asks whether or not there exists a hypergraph (or, equivalently, a simplicial complex) such that the Stanley-Reisner ideal of its closure has a linear resolution over every field, but which is not a generalized chordal hypergraph.

\begin{figure}[h!]
{\centering
   \includegraphics[height=1.2in]{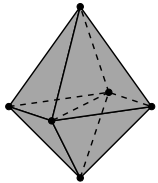}
    \caption {$2$-chorded simplicial complex which is not ``chordal''.} \label{fig:chorded_not_chordal}

}
\end{figure}


\section{The converse: which complexes have a linear resolution?} \label{sec:does_chorded_have_LR}

The converses to parts 1 and 2 of Theorem \ref{thm:one_dir} do not hold. Consider the following counterexample to the converse of part 1.

\begin{example}
The simplicial complex $\Gamma$ in Figure \ref{fig:sphere_tetra} is a triangulated sphere with a hollow tetrahedron suspended within it from four pairwise non-adjacent vertices.  The sphere is not an orientably-vertex-minimal $2$-dimensional cycle as the hollow tetrahedron is an orientable $2$-dimensional cycle on a strict subset of its vertices.  Thus $\Gamma$ is orientably-$2$-cycle-complete as its only orientably-vertex-minimal $2$-dimensional cycle, the tetrahedron, is $2$-complete.  The complex $\Delta_2(\Gamma)$ adds all possible $1$-faces to $\Gamma$ and adds the $3$-face consisting of the four vertices of the tetrahedron.  It is clear that the $2$-faces of the sphere in $\Delta_2(\Gamma)$ form a homological $2$-cycle which is not a $2$-boundary in $\Delta_2(\Gamma)$ for any field $k$ and therefore $\tilde{H}_2(\Delta_2(\Gamma);k) \neq 0$.  Hence $\N(\Delta_2(\Gamma))$ does not have a linear resolution.
\end{example}

Now consider the following counterexample to the converse of Theorem \ref{thm:one_dir} part 2.

\begin{example} \label{ex:chorded_no_LR}
Let $\Gamma$ be the pure $3$-dimensional simplicial complex on vertices $x_0,\ldots,x_6$ that is obtained from $\Lambda_7^3$ by removing the following five facets:
\[
\begin{matrix}
  x_0x_1x_5x_6 & x_0x_2x_5x_6 & x_0x_3x_5x_6, \\
  x_0x_4x_5x_6 & x_1x_2x_3x_4
\end{matrix}
\]

The facets of $\Gamma$ are:
\[
\begin{matrix}
  x_0x_1x_2x_3 & x_0x_1x_2x_4 & x_0x_1x_2x_5 & x_0x_1x_2x_6 & x_0x_1x_3x_4 & x_0x_1x_3x_5 \\
  x_0x_1x_3x_6 & x_0x_1x_4x_5 & x_0x_1x_4x_6 & x_0x_2x_3x_4 & x_0x_2x_3x_5 & x_0x_2x_3x_6 \\
  x_0x_2x_4x_5 & x_0x_2x_4x_6 & x_0x_3x_4x_5 & x_0x_3x_4x_6 & x_1x_2x_3x_5 & x_1x_2x_3x_6 \\
  x_1x_2x_4x_5 & x_1x_2x_4x_6 & x_1x_2x_5x_6 & x_1x_3x_4x_5 & x_1x_3x_4x_6 & x_1x_3x_5x_6 \\
  x_1x_4x_5x_6 & x_2x_3x_4x_5 & x_2x_3x_4x_6 & x_2x_3x_5x_6 & x_2x_4x_5x_6 & x_3x_4x_5x_6
\end{matrix}
\]
The complex $\Gamma$ is $3$-chorded but the Stanley-Reisner ideal of the $4$-dimensional simplicial complex $\Delta_3(\Gamma)$ does not have a linear resolution over $\Z_2$.  In fact, the pure $4$-skeleton of $\Delta_3(\Gamma)$ is a $4$-dimensional cycle with no chord set which is not $4$-complete.  Therefore $\tilde{H}_4(\Delta_3(\Gamma);\Z_2)\neq 0$.
\end{example}

All counter-examples to the converse of Theorem \ref{thm:one_dir} part 2 share a specific property.  The $d$-closures of these $d$-chorded complexes contain face-minimal non-$n$-complete $n$-dimensional cycles having complete $1$-skeletons and having no chord sets, where $n>d$.  It is this feature which prevents the desired linear resolution by introducing homology on a level higher than the dimension of the original complex.  We prove this in \cite{ConFar13} and establish a necessary and sufficient condition for a monomial ideal to have a linear resolution over fields of characteristic $2$.

Although it is not the case that the $d$-closure of all $d$-chorded complexes have Stanley-Reisner ideals with linear resolutions over fields of characteristic $2$, this does hold for the smaller class of $d$-dimensional trees.  As we will see in the proof of the following theorem the $d$-closures of these complexes have no $n$-dimensional cycles for $n\geq d$.  It follows that, over fields having characteristic $2$, all upper-level homologies in the closure are zero.

\begin{thm} \label{thm:tree_has_lin_res}
If $\Gamma$ is a $d$-dimensional tree then $\N(\Delta_d(\Gamma))$ has a $(d+1)$-linear resolution over any field of characteristic $2$.
\end{thm}

\begin{proof}
By Theorem \ref{thm:Frob_hom_char} we need to show that, for any field $k$ of characteristic $2$, $\tilde{H}_i(\Delta_d(\Gamma)_W;k) = 0$ for all $i \neq d-1$ and all $W \subseteq V(\Gamma)$.  However, it is not hard to see that the pure $d$-skeleton of any induced subcomplex of a $d$-dimensional tree is also a $d$-dimensional tree and so by Lemma \ref{lem:delta_operation_induced} we need only show that $\tilde{H}_i(\Delta_d(\Gamma);k)=0$ for all $i \neq d-1$.

Since $\Delta_d(\Gamma)$ has all possible faces of dimension less than $d$, by its definition, we know that $\tilde{H}_i(\Delta_d(\Gamma);k) = 0$ for all $i < d-1$.  Since $\Gamma$ has no $d$-dimensional cycles, neither does $\Delta_d(\Gamma)$ and so by Theorem \ref{thm:cycle_hom_z2} we must have $\tilde{H}_d(\Delta_d(\Gamma);k) = 0$.

We claim that $\Delta_d(\Gamma)$ has no faces of dimension greater than $d$.  If $\Delta_d(\Gamma)$ contains a face of dimension greater than $d$ then it must contain a face of dimension $d+1$.  Such a face exists in $\Delta_d(\Gamma)$ only when all subsets of its vertices of size $d+1$ are faces of $\Gamma$.  But these $d$-faces of $\Gamma$ then form a $d$-dimensional cycle in $\Gamma$ by Proposition \ref{prop:smallest_cycle}.  This is a contradiction since $\Gamma$ contains no $d$-dimensional cycles and so $\Delta_d(\Gamma)$ contains no faces of dimension greater than $d$.  Hence it must be the case that $\tilde{H}_i(\Delta_d(\Gamma);k) = 0$ for all $i > d$.

Therefore $\tilde{H}_i(\Delta_d(\Gamma);k) = 0$ for all $i \neq d-1$ and so, by Theorem \ref{thm:Frob_hom_char}, $\N(\Delta_d(\Gamma))$ has a $(d+1)$-linear resolution over $k$.
\end{proof}

Theorem \ref{thm:tree_has_lin_res} is equivalent to Theorem \ref{thm:trees_intro} given in Section \ref{sec:intro} since the facet ideal of the $d$-complement of a simplicial complex $\Gamma$ is the Stanley-Reisner ideal of the $d$-closure of $\Gamma$.


\section{Chorded complexes and componentwise linear ideals} \label{sec:chorded}
For square-free monomial ideals whose generators are not all of the same degree, the property of being {\bf componentwise linear} is analogous to having a linear resolution.

\begin{definition}[{\bf componentwise linear}]
A square-free monomial ideal $I$ is {\bf componentwise linear} over the field $k$ if $I_{[d]}$ has a linear resolution over $k$ for all $d$, where $I_{[d]}$ is the ideal generated by the square-free monomials in $I$ of degree $d$.
\end{definition}

The Stanley-Reisner complex of such an ideal will not be the closure of a pure simplicial complex.  However, we may still observe some combinatorial properties of this non-pure complex itself.  We introduce the notion of a {\bf chorded} complex to restrict cycles on all dimensions of the simplicial complex.

\begin{definition}[{\bf chorded}]
A simplicial complex $\Gamma$ is {\bf chorded} if $\Gamma^{[d]}$ is $d$-chorded for all $d \leq \dim \Gamma$.
\end{definition}

Before showing that such complexes result from componentwise linear ideals, we require the following lemma.

\begin{lem} \label{lem:equiv_of_skels}
Given a simplicial complex $\Gamma$ we have
\[
    \Gamma^{[d-1]} = \N(\N(\Gamma)_{[d]})^{[d-1]}.
\]
\end{lem}

\begin{proof}
Let $F \in \Gamma^{[d-1]}$.  Then $x^F \notin \N(\Gamma)$.  Hence $x^F \notin \N(\Gamma)_{[d]}$.  Therefore $F \in \N(\N(\Gamma)_{[d]})$, the Stanley-Reisner complex of $\N(\Gamma)_{[d]}$, and $F \in \N(\N(\Gamma)_{[d]})^{[d-1]}$ because $|F|=d$.

Conversely, let $F \in \N(\N(\Gamma)_{[d]})^{[d-1]}$.  Then $F \in \N(\N(\Gamma)_{[d]})$ and so $x^F \notin \N(\Gamma)_{[d]}$.  Since $|F|=d$, this means that $x^F \notin \N(\Gamma)$.  Therefore $F \in \Gamma$ and so $F \in \Gamma^{[d-1]}$.
\end{proof}

\begin{thm} \label{thm:cwl_is_chorded}
If $\N(\Gamma)$ is componentwise linear over every field $k$ then $\Gamma$ is chorded.
\end{thm}

\begin{proof}
Since $\N(\Gamma)_{[d]}$ has a linear resolution over all fields $k$, we have that for all $d$,
\[
    \N(\N(\Gamma)_{[d]}) = \Delta_{d-1}\left(\N(\N(\Gamma)_{[d]})^{[d-1]}\right)
\]
and $\N(\N(\Gamma)_{[d]})^{[d-1]}$ is $(d-1)$-chorded by Corollary \ref{cor:one_dir_over_all}.  Hence by Lemma \ref{lem:equiv_of_skels} we know that $\Gamma^{[d-1]}$ is $(d-1)$-chorded for all $d\leq \dim \Gamma +1$.  Hence $\Gamma$ is chorded.
\end{proof}

\bibliography{chorded_paper_bib}

\end{document}